\documentclass[12pt,a4paper,reqno]{amsart}
\usepackage{amsfonts,amsthm,amsmath,amscd}
\usepackage{booktabs}
\usepackage[utf8]{inputenc}
\usepackage{t1enc}
\usepackage[mathscr]{eucal}
\usepackage{indentfirst}
\usepackage{tikz}
\usetikzlibrary{shapes.geometric, positioning, calc}
\usetikzlibrary{automata,arrows,positioning,calc}
\usetikzlibrary{positioning}
\usepackage{caption}
\usepackage{subcaption}
\usepackage[font=small,labelfont=bf]{caption}
\usepackage{graphicx,graphics,pict2e}
\usepackage{tkz-berge}
\usetikzlibrary{arrows.meta} 
\usetikzlibrary{decorations.pathmorphing,decorations.markings} 
\usepackage{float}
\usepackage{epic}
\usepackage{lipsum}
\usepackage{stmaryrd}
\usepackage{authblk}
\usepackage[symbol]{footmisc}
\usepackage[normalem]{ulem}
\numberwithin{equation}{section}
\usepackage[margin=2.9cm]{geometry}
\usepackage{epstopdf}
\RequirePackage{doi}
\usepackage{hyperref}
\allowdisplaybreaks
\usepackage{cite}
\usepackage{multirow}
\usepackage{multicol}
\usepackage{verbatim,amssymb,amsfonts}
\usepackage{mathrsfs}
\usepackage{mathtools}
\DeclarePairedDelimiter\floor{\lfloor}{\rfloor}
\usepackage{tikz-cd}
\usepackage{indentfirst}
\usepackage{slashed}
\usepackage{thmtools}
\usepackage{setspace}
\usepackage{enumerate}
\usepackage{accents}

\usetikzlibrary{arrows.meta, bending}

\theoremstyle{plain}
\newtheorem{theorem}{Theorem}[section]

\newtheorem{proposition}[theorem]{Proposition}

 \theoremstyle{definition}
\newtheorem{definition}[theorem]{Definition}
\newtheorem{remark}[theorem]{Remark}
\newtheorem{example}[theorem]{Example}

\DeclarePairedDelimiterX{\inp}[2]{\langle}{\rangle}{#1, #2}
\newcommand{\<}{\langle}
\renewcommand{\>}{\rangle}

\newcommand{\cL}{{\mathcal L}}

\newcommand{\cV}{{\mathcal V}}

\newcommand{\diam}{{\rm{diam}}}

\newcommand{\ba}{\begin{eqnarray}}
\newcommand{\na}{\end{eqnarray}}
\newcommand{\ban}{\begin{eqnarray*}}
\newcommand{\nan}{\end{eqnarray*}}

\newcommand{\R}{{\mathbb R}}
\newcommand{\Z}{{\mathbb Z}}

\renewcommand{\thefootnote}{\fnsymbol{footnote}}

\makeatletter
\g@addto@macro{\endabstract}{\@setabstract}
\newcommand{\authorfootnotes}{\renewcommand\thefootnote{\@fnsymbol\c@footnote}}%
\makeatother

\makeatletter
\@namedef{subjclassname@2020}{%
  \textup{2020} Mathematics Subject Classification}
\makeatother

\title[]{Steinerberger Curvature on Digraphs: \\Discrete Bonnet--Myers and Lichnerowicz Theorems}


\begin{document}

\begin{center}
    \vspace{-1cm}
	\maketitle
	
	\normalsize
    \authorfootnotes
    Kevin Fung, Johnny Lim\footnote[1]{Corresponding author.}
	\par \bigskip

        \small{School of Mathematical Sciences, Universiti Sains Malaysia, Penang, Malaysia}\par \bigskip
\end{center}

\address{School of Mathematical Sciences, Universiti Sains Malaysia, Penang, Malaysia
}
\email{fungkevin@student.usm.my}
\email{johnny.lim@usm.my}

\begin{abstract}
Steinerberger curvature encodes the global distance geometry of a graph through an equilibrium measure. In this paper, we derive explicit curvature formulas for undirected Cayley graphs of dihedral groups $D_n$ and generalized quaternion groups $Q_{4m}$. We then extend Steinerberger curvature to strongly connected simple digraphs by introducing in-curvature and out-curvature, reflecting the asymmetry of directed distances. For these directed curvatures, we establish structural properties, including negativity criteria and a permutation relation between in- and out-curvature. Our main results are directed analogues of the Bonnet--Myers, Cheng and Lichnerowicz theorems, together with reverse Bonnet--Myers inequalities for directed diameter and out-radius, and an upper bound for in-radius in terms of total curvature.
\end{abstract}

\section{Introduction}
\label{sec1}

Curvature is one of the central ideas in geometry, and several graph-theoretic analogues have been developed to transfer geometric intuition to discrete spaces, see, e.g.,  \cite{linluyau2011Ricci,lott2009Riccimetric, ollivier2009Riccimetric,sturm2006geometryI,sturm2006geometryII}. Among these notions, Steinerberger curvature is distinguished by its global and metric nature.
Let $G=(V, E)$ be an undirected simple graph with distance matrix $D=(d(v,w))_{v,w\in V}$. The Steinerberger curvature of $G$ is defined as a \textit{measure} $\mu:V \rightarrow \R$ satisfying 
\begin{equation}
\label{eq_steinerberger}
    \sum_{w\in \cV} d(v,w)\mu(w)=|V|,
\end{equation}
for every $v\in V,$ or equivalently, $\mu$ solves the system $D\mu =n\cdot\textbf{1}$, where $n=|V|$ and $\textbf{1}$ is the $n\times 1 $ column vector with all entries being one. If the system does not admit a solution, the curvature is defined as $\mu= D^{\dagger}(n \cdot \textbf{1})$ where $D^{\dagger}$ is the Moore-Penrose inverse of $D$. These are called distance exceptional graphs. We refer the readers to \cite{steinerberger2023curvature} for further details.

Although Steinerberger curvature is relatively recent, it already reveals connections with other notions of discrete curvature. For example, the Steinerberger curvature of complete graphs $K_n$ and hypercube graph $Q_n$ agrees with Lin--Lu--Yau Ricci curvature, while for Cocktail Party graphs, Johnson graphs and Demi-cubes graphs, it coincides with Ollivier-Ricci curvature. 
Several structural results are also known. 
It has been shown in \cite{chen2025steinerberger} that nonnegative Steinerberger curvature is largely preserved under graph operations such as bridging (adding an edge between two graphs), merging (bridging two graphs and then contracting the edge), and cutting (removing an edge that disconnects the graph into two connected components), except possibly at a small number of vertices. These operations also provide methods for constructing distance exceptional graphs \cite{robertson2025distexceptional}, and explicit curvature formulas are known for graphs obtained by bridging two graphs \cite{cushing2025note_steinerbeger}. 

In \cite{mizukaiRicci_Cayley2024}, Iwao and Akifumi computed explicitly the Lin--Lu--Yau curvature of undirected Cayley graphs of dihedral groups $D_n$, generalized quaternion groups $Q_{4m}$ and cyclic groups $\Z_n$. We have recently generalized these to the case of digraphs, cf. \cite{kevin2026lly_directed}. These developments naturally lead to two immediate questions: 
\begin{enumerate}
    \item Does the Steinerberger curvature of these undirected Cayley graphs agree with their Lin--Lu--Yau Ricci curvature?
    \item What is the appropriate directed analogue of Steinerberger curvature?
\end{enumerate}

Motivated by these questions, this paper studies  Steinerberger curvature in two directions. First, we provide explicit formulas for the Steinerberger curvature of undirected Cayley graphs of dihedral groups $D_n$ and generalized quaternion group $Q_{4m},$ where the generating sets consist of all generators of the respective groups together with their inverses. A comparison with Lin--Lu--Yau curvature is made, highlighting the fundamental difference between two curvatures. Second, we introduce a \textit{directed analogue} of  Steinerberger curvature. Since directed distances are generally non-symmetric, we define \textit{in-curvature} and \textit{out-curvature} for digraphs. We do not aim to compute the Steinerberger curvature for directed Cayley graphs as we did for the undirected case. Rather, we showed that this framework allows classical comparison-type results of Riemannian geometry such as Bonnet--Myers Theorem, Cheng's Theorem, Lichnerowicz Theorem and reverse Bonnet--Myers Theorem to be formulated in the directed setting. 

The paper is organized as follows. Sect. \ref{sect:preliminaries} recalls the necessary preliminaries. Sect. \ref{sect:undirectedcayley} computes the Steinerberger curvature of undirected Cayley graphs of the dihedral groups $D_n$ and generalized quaternion groups $Q_{4m}$, followed by a comparison with Lin--Lu--Yau curvature. Sect. \ref{sect:curv_directed} studies the directed case, beginning with sufficient conditions for  negative in- and out-curvatures in Proposition \ref{prop_suff_negcurv}.
We then prove the Bonnet--Myers and Cheng's Theorem in Theorem \ref{theo:bonnet-myers}, 
and the Lichnerowicz Theorem in Theorem \ref{theo:lichnerowicz}, and a variational theorem in Theorem \ref{theorem_variation}. 
This leads to reverse Bonnet--Myers Inequality for the directed diameter and out-radius in Theorem \ref{theo_reverse_bonnet} and Theorem \ref{theo:reversebonnet_outradius}.
Finally, Theorem \ref{theo:inradius_ub} establishes an upper bound for the in-radius in terms of total curvature.

\section{Preliminaries}
\label{sect:preliminaries}

In this section, we list the necessary preliminaries in this article.

\begin{definition}\cite{jorgen2009digraphs_book}
Let $G=(V,E)$ be a digraph, where $V$ and $E$ denote the set of vertices and arcs, respectively. 
    \begin{enumerate}
        \item For any two vertices $v,w \in V$, if there is a directed path from $v$ to $w$, then $w$ is said to be \textit{reachable} from $v$. A digraph $G$ is  \textit{strongly connected} if any two vertices in $V$ are reachable. The distance $d(v,w)$ from $v$ to $w$ is the length of a shortest directed path from $v$ to $w$.

        \item For any vertex $v\in V$, its \textit{out-degree} $d_v^{\text{out}}$ is the number of arcs leaving $v$. A digraph $G$ is \textit{locally finite} if every vertex has a finite out-degree.


        \item A digraph $G$ is \textit{$r$-regular} if all vertices have the same out-degree  $r$.

        \item A digraph $G$ is \textit{simple} if it has no multiple arcs and loops.
    \end{enumerate}
\end{definition}

\begin{definition}\cite{jorgen2009digraphs_book}
\label{def2.3}
    Let $G=(V,E)$ be a digraph and let $X,Y \subseteq V$. The \textit{distance} from $X$ to $Y$ is defined as
    \begin{equation*}
        \text{dist}(X,Y) = \max_{x\in X ,y\in Y} \{ d(x,y)\}.
    \end{equation*}
    \begin{enumerate}[(i)]
        \item The directed \textit{diameter} of $G$ is defined as
            \begin{equation*}
            \overrightarrow{\diam} (G)= \text{dist} (V,V).
            \end{equation*} 
        
        \item The \textit{in-radius} of $G$
        is defined as
            \begin{equation*}
            \mathrm{rad}^-(G) = \min_{x \in V} \{ \text{dist}(V,x)\}.
            \end{equation*}
        
        \item The \textit{out-radius} of $G$
        is defined as
            \begin{equation*}
            \mathrm{rad}^+(G) = \min_{x \in V} \{ \text{dist}(x,V) \}.
            \end{equation*}
    \end{enumerate}
\end{definition}

\begin{definition}
    \cite{grossman1964groups}
    \label{directed Cayley graphs}
    Let $G$ be a group $G$ and $S$ be a subset of $G-\{e\}$. A \textit{directed Cayley graph} $\Gamma (G, S)$ is defined as a simple directed graph with vertex set $G$ and arcs of the form $(g,gs)$ for every $g \in G$ and $s \in S$.
\end{definition}

\begin{definition}
    \cite{grossman1964groups}
    \label{undirected Cayley graphs}
    Let $G$ be a group $G$ and $S$ be a subset of $G-\{e\}$. An \textit{undirected Cayley graph} $\Gamma (G, S)$ is defined as a directed Cayley graph such that
    $S$ is symmetric (inverse-closed), i.e. $S=S^{-1}$, and symmetric arcs between two vertices are identified as one edge.
\end{definition}

\begin{remark}
\begin{enumerate}
    \item In Def. \ref{directed Cayley graphs} (resp. Def. \ref{undirected Cayley graphs}), if $S$ generates $G$, i.e. $G=\langle S \rangle$, then $G$ is strongly connected (resp. connected).
    
    \item The Cayley graph $\Gamma(G,S)$ is always regular of degree $|S|$.

    \item If $S$ contains the unit $e$, then every vertex $g \in G$ has one self-loop because $(g,ge)=(g,g)$. 
\end{enumerate}
\end{remark}

\begin{definition}{(\cite{bjorner2006combinatorics}, \cite{Johnson1980groupspresentation})}
    \begin{enumerate}
        \item For $n\geq 3$, the dihedral groups, $D_n$ are defined as 
        \begin{equation*}
            D_n=\left< a,b \mid a^n=b^2=e \text{ and } ba=a^{n-1}b \right>.
        \end{equation*}

        \item For $m\geq 2$, the generalized quaternion groups, $Q_{4m}$ are defined as 
        \begin{equation*}
            Q_{4m}= \left< a,b \mid a^{2m}=e, b^2=a^m \text{ and } b^{-1}ab=a^{-1} \right>.
        \end{equation*}
    \end{enumerate}
\end{definition}


\begin{definition}
\cite{steinerberger2023curvature}
\label{def_steinerberger_curv}
    Let $G=(V, E)$ be a finite simple undirected connected graph with $|V|=n$. Let $D=(d(v_i, v_j))_{1 \leq i,j \leq n}$ be the distance matrix of $G$ and $\textbf{1}$ is the $n\times 1 $ column vector with all entries being one. The Steinerberger curvature of $G$ is the measure $\mu_G:V \rightarrow \R$ such that one of the following is satisfied  
    \begin{enumerate}
        \item If the equation
    \begin{equation}
        \label{eq_curvature}
        DK= n\cdot \textbf{1},
    \end{equation}
    has a unique solution, then $\mu_G$ is defined to be that solution.
    
        \item If Equation \eqref{eq_curvature} has more than one solutions, then $\mu_G$ is defined to be the one that achieves $\max_{K} \{ \min_{1\leq i \leq n} K_i \}$, where $K=(K_1, \ldots, K_n)$ runs over all solutions satisfying Equation \eqref{eq_curvature}.

        \item If Equation \eqref{eq_curvature} has no solution, then $\mu_G$ is defined as 
        \begin{equation}
            \mu = D^\dagger (n\cdot \textbf{1}),
        \end{equation}
        where $D^\dagger$ is the Moore-Penrose pseudo-inverse of $D$, which always exists and is unique.
    \end{enumerate}
\end{definition}

For vertex transitive graphs, its Steinerberger curvature is constant and is given by the following proposition. 
\begin{proposition}\cite[Proposition 2]{steinerberger2023curvature}
\label{prop_curvature_vertex_transitive}
    If $G$ is vertex transitive, then it has a constant curvature $K>0$ given by 
    \begin{equation}
        K =\left( \frac{1}{n} \sum_{i=1} ^n d(v,v_i)\right)^{-1},
    \end{equation}
    for any $v\in \cV$.
\end{proposition}

    For a digraph $G=(V,E)$, a \textit{transition probability matrix} is a matrix $P\in \R^{n\times n}$ with entries $P(u,v)$ which denote the probability of going from $u$ to $v$. In particular, $P(u,v) >0$ if and only if arc $(u,v)$ exists. Moreover, $\sum_{v\in V} P(u,v)=1$ but it is not necessary that $\sum_{u\in V} P(u,v)=1$, cf. \cite{chung2005laplcian_cheeger_directed}. In this article, we shall consider $P$ with entries 
    \begin{equation}
        P(u,v)=
        \begin{cases}
            \dfrac{1}{d^\text{out}_u}, &\quad \text{if $(u,v)\in E$,} \\
            0, &\quad \text{otherwise.}
        \end{cases}
    \end{equation}
    Since $P$ is nonnegative and irreducible, by Perron-Frobenius Theorem \cite{perron1907, frobenius1912}, $P$ has a unique positive left eigenvector $\phi$ corresponds to eigenvalue $\rho=1$, i.e.
      $  \phi P= \phi.$
    Throughout, $\phi$ is normalized so that $\sum_{v\in V} \phi(v)=1$. We can now state the following definition:

\begin{definition}\cite{chung2005laplcian_cheeger_directed}
    Let $\Phi$ be the diagonal matrix with entries $\Phi(v,v)= \phi(v)$. The \textit{Laplacian} of a digraph $G$ is defined as 
    \begin{equation}
        \cL =I- \dfrac{\Phi^\frac{1}{2}P \Phi^{-\frac{1}{2}}+\Phi^{-\frac{1}{2}}P^T \Phi^{\frac{1}{2}} }{2}.
    \end{equation}
\end{definition}

\begin{proposition}\cite[Corollary 1]{chung2005laplcian_cheeger_directed}
     Let $0=\lambda_0\leq\lambda_1\leq\ldots \leq \lambda_{n-1}$ be the eigenvalues of $\cL$. Then,
     \begin{equation}
     \label{eq_lambda_1_inf}
         \lambda_1 = \inf_{\substack{f:V\rightarrow \R \\ \sum_{x\in V}f(x)\phi(x)=0}} \dfrac{\sum_{(u,v)\in E}(f(u)-f(v) )^2 \phi(u)P(u,v) }{2\sum_{v\in V} f(v)^2\phi(v)}.
     \end{equation}
\end{proposition}

\section{Steinerberger Curvature of Undirected Cayley Graphs}
\label{sect:undirectedcayley}

It is well known that undirected Cayley graphs are vertex-transitive. This enables us to apply Proposition \ref{prop_curvature_vertex_transitive} to obtain an explicit formulas of the Steinerberger curvature of the Cayley graph $\Gamma(G,S)$ for $G=D_n$ and $G=Q_{4m}.$  

\begin{proposition}
\label{prop_dn}
    Let $D_n$ be the dihedral group with generating set $S=\{a, a^{-1}, b, b^{-1} \}$. Then, it has constant Steinerberger curvature of the form:
    \begin{equation}
        K = \frac{2n}{2\floor{\frac{n}{2}} (\floor{\frac{n}{2}}+1) + n^2-2n \floor{\frac{n}{2}}},
    \end{equation}
    where $\floor{\frac{n}{2}}$ denotes the greatest integer of $\frac{n}{2}$.
\end{proposition}

\begin{proof}
    The group $D_n$ can be represented by $D_n =\{e, a, \ldots, a^{n-1}, b, ab, \ldots\ , a^{n-1}b \}$.
    For $k=0, \ldots, n-1$, since $a^k a^{n-k}=e$, we have 
    \[
    d(e, a^k)=\min\{k, n-k\} \quad \text{ and } \quad 
    d(e, a^kb)=\min\{k, n-k\}+1.
    \]
    Therefore, 
    \begin{align*}
        \sum_{i=1}^{2n} d(e, g_i) &= \sum_{k=0}^{n-1} d(e, a^k) + \sum_{k=0}^{n-1} d(e, a^kb) \\ 
        &= \sum_{k=0}^{n-1} \min \{k, n-k\} + \sum_{k=0}^{n-1} \left[ \min \{k, n-k\} + 1 \right] \\
        &= 2 \sum_{k=0}^{\floor{\frac{n}{2}}} k + 2\sum_{k=\floor{\frac{n}{2}}+1}^{n-1} (n-k) + n \\ 
        &= 2\frac{\floor{\frac{n}{2}}(\floor{\frac{n}{2}}+1)}{2} + 2n\left(n-\floor*{\frac{n}{2}}-1 \right) -2\sum_{k=\floor{\frac{n}{2}}+1}^{n-1} k +n \\
        &= \floor*{\frac{n}{2}}\left(\floor*{\frac{n}{2}}+1 \right) + 2n\left(n-\floor*{\frac{n}{2}}-1 \right) -2 \left[ \frac{(n-1)n}{2} - \frac{\floor{\frac{n}{2}}\left(\floor{\frac{n}{2}}+1 \right)}{2} \right] + n \\ 
        &= 2\floor*{\frac{n}{2}} \left(\floor*{\frac{n}{2}}+1 \right) + n^2-2n \floor*{\frac{n}{2}}.
    \end{align*}
     By Proposition \ref{prop_curvature_vertex_transitive}, it follows that 
         \[
         K = \left( \frac{1}{2n} \sum_{i=1} ^{2n} d(e,g_i)\right)^{-1} 
         = \frac{2n}{2\floor{\frac{n}{2}} (\floor{\frac{n}{2}}+1) + n^2-2n \floor{\frac{n}{2}}}.
         \]
\end{proof}

\begin{proposition}
\label{prop_q4m}
    Let $Q_{4m}$, $m \geq 2$ be the generalized quaternion group with generating set $S=\{ a, b, a^{-1}, b^{-1}\}$. Then, it has constant Steinerberger curvature of the form
    \begin{equation}
        K=\frac{4}{2m+1}.
    \end{equation}
\end{proposition}

\begin{proof}
    The group $Q_{4m}$ can be represented as 
    $$Q_{4m}= \{e, a, \ldots, a^{2m-1}, b, ab, \ldots, a^m b, \ldots, a^{2m-1}b \}.$$
    Note that $a^{-1}=a^{2m-1}$ and $b^{-1}= a^m b$. Moreover, observe that
    \begin{enumerate}
        \item $d(e, a^m b)=1$.
        \item $d(e, a^k)=\min \{k, 2m-k \}$ for $k=0, \ldots, 2m-1$.
        \item $d(e, a^k b)=\min \{k, 2m-k\}+1$ for $k=0, \ldots, \hat{m}, \ldots, 2m-1$, where $\hat{m}$ means that $m$ is omitted.
    \end{enumerate}

    Hence, 
    \begin{align*}
        \sum_{k=1}^{4m} d(e, g_k) &= d(e, a^m b) + \sum_{k=0}^{2m-1} d(e, a^k) + \sum_{k=0}^{m-1} d(e, a^k b) +\sum_{k=m+1}^{2m-1} d(e, a^k b) \\
        &=1+ \sum_{k=0}^{2m-1} \min \{k, 2m-k \} 
        + \sum_{k=0}^{m-1} (\min \{k, 2m-k\}+1 ) + \sum_{k=m+1}^{2m-1} (\min \{k, 2m-k\}+1 ) \\
        &= 1 +  \sum_{k=0}^{m} k + \sum_{k=m+1}^{2m-1} (2m-k) + \sum_{k=0}^{m-1} (k+1) + \sum_{k=m+1}^{2m-1} (2m-k+1) \\
        &= 1+\frac{m(m+1)}{2}+\frac{(m-1)m}{2}+  \frac{m(m+1)}{2} + \left( \frac{m(m+1)}{2}-1 \right)\\
        &= m(2m+1).
    \end{align*}
    It follows that 
    \[
    K= \left(\frac{m(2m+1)}{4m} \right)^{-1} = \frac{4}{2m+1}. \qedhere
    \]
\end{proof}

\begin{remark}
 In Propositions \ref{prop_dn} and \ref{prop_q4m}, the Steinerberger curvature is constant and positive. In particular,  we have $\mu_G = K \textbf{1}$ where $K$ is the common curvature at every vertex. Thus, the \textit{total} curvature is $\|\mu_G \|_{\ell_1} = \sum_{v \in V} \mu_G(v).$ 
 \begin{enumerate}[(i)]
 \item  For the dihedral group $D_n,$ since $|D_n|=2n,$ and since 
 \[
 K = 
 \begin{cases}
 \dfrac{4}{n+2}, &n=2k,\\[10pt]
 \dfrac{4n}{n^2+2n-1}, & n=2k+1,
 \end{cases} 
 \]
 we obtain
 \[
 \|\mu_{\Gamma(D_n,S)} \|_{\ell_1} = 2nK = 
  \begin{cases}
  \dfrac{8n}{n+2}, &n=2k,\\[10pt]
  \dfrac{8n^2}{n^2+2n-1}, & n=2k+1.
  \end{cases} 
 \] 
 \item For the dihedral group $Q_{4m},$ since $|Q_{4m}|=4m,$
 \[
\|\mu_{\Gamma(Q_{4m},S)} \|_{\ell_1} = 4mK = \frac{16m}{2m+1}.
 \]
 \end{enumerate}
 For both cases, it is clear that for large group orders, both curvatures tend to the same limiting value despite that they are different non-abelian groups:
 \[
 \lim_{n \to \infty} \|\mu_{\Gamma(D_n,S)} \|_{\ell_1} = 8 =\lim_{m \to \infty} \|\mu_{\Gamma(Q_{4m},S)} \|_{\ell_1} .
 \]
\end{remark}

We end this section by comparing these curvatures with the Lin--Lu--Yau Ricci curvature obtained in \cite{mizukaiRicci_Cayley2024}, see Tables \ref{tab:dn} and \ref{tab:q4m}.

\begin{table}
\centering
\makebox[0pt][c]{\parbox{1.1\textwidth}{%
    \begin{minipage}[b]{0.6\hsize}\centering
    \renewcommand{\arraystretch}{1.4}
\begin{tabular}{|c|c|c|c|c|}
\hline
\multicolumn{5}{|c|}{$\Gamma(D_n, S),\quad S=\{a,a^{-1},b=b^{-1}\}$} \\
\hline
$\kappa(x,y)$ & $n=3$ & $n=4$ & $n=5$ & $n\geq 6$ \\
\hline
$\kappa(e,a)$ & $1$ & $\frac{2}{3}$ & $\frac{1}{3}$ & 0 \\
\hline
$\kappa(e,b)$ & $\frac{2}{3}$ & $\frac{2}{3}$ & $\frac{2}{3}$ & $\frac{2}{3}$ \\
\hline
\end{tabular}
\caption{$\kappa(e,-)$ for $\Gamma(D_n,S)$}
\label{tab:dn}
\end{minipage}
    \hspace{-3cm}    
\begin{minipage}[b]{0.6\hsize}\centering
\renewcommand{\arraystretch}{1.4}
\begin{tabular}{|c|c|c|c|}
\hline
\multicolumn{4}{|c|}{$\Gamma(Q_{4m}, S),\quad S=\{a,a^{-1},b,b^{-1}\}$} \\
\hline
$\kappa(x,y)$ & $m=2$ & $m=3$ & $m\geq 4$ \\
\hline
$\kappa(e,a)$ & $\frac{1}{2}$ & $\frac{1}{4}$ & $0$ \\
\hline
$\kappa(e,b)$ & $\frac{1}{2}$ & $\frac{1}{2}$ & $\frac{1}{2}$ \\
\hline
\end{tabular}
\caption{$\kappa(e,-)$ for $\Gamma(Q_{4m},S)$}
\label{tab:q4m}
\end{minipage}
}}
\end{table}

Although $\mu_{\Gamma(D_4, S)}=\frac{2}{3}\mathbf{1_4}$ and $\kappa(e,a)=\kappa(e,b)=\frac{2}{3}$, we see that the Steinerberger curvature for $\Gamma(D_n,S)$ and $\Gamma(Q_{4m},S)$ are, in general, different from that of Lin--Lu--Yau for other values of $n$. Another observation is that for $\Gamma(D_n,S)$, $\lim_{n \to \infty} \mu_{\Gamma(D_n,S)}= \mathbf{0_n}$ but $\kappa(e,b)=\frac{2}{3}\neq 0$ for $n\geq 6$. Similar observation can be made for $ \mu_{\Gamma(Q_{4m},S)}$.

These examples show that the Steinerberger curvature and Lin--Lu--Yau Ricci curvature do not agree in general, even for highly symmetric Cayley graphs. 
The underlying reason is probably due to the distinctive difference between them: Steinerberger curvature is a \textit{vertex-based} invariant determined by the global distance matrix, whereas Lin--Lu--Yau Ricci curvature is an \textit{edge-based} invariant determined by optimal transport between local probability measures.
The partial agreements observed above suggest that the two curvature notions may still be related in a certain way. Establishing a precise relation between these curvatures remain a natural direction for future work.

\section{Steinerberger Curvature of Directed Graphs}
\label{sect:curv_directed}
For a digraph $G=(V,E)$ with $|V|=n$, since the distance matrix $D$ is not symmetric in general, it is natural to consider two notions of curvature that correspond to the in- and out-direction of arcs. With this, we provide a generalization of Steinerberger curvature to the case of strongly connected simple digraphs, for which we will refer it as \textit{\textbf{digraphs}} henceforth. 

\begin{definition}[Steinerberger curvature for digraphs]
\label{def_steinerbergercurv_directed}
    Let $G$ be a digraph. Define the \textit{out-curvature }$\mu^{\text{out}}_G$ of $G$ as either one of the following: 
        \begin{enumerate}
        \item The unique solution to $DK=n \cdot \textbf{1}$.

         \item If $DK=n\cdot \textbf{1}$ has more than one solutions, then $\mu^{\text{out}}_G$ is defined as the one such that $\max_{K} \{ \min_{1\leq i \leq n} K_i \}$ is achieved, where $K=(K_1, \ldots, K_n)$ runs over solutions satisfying $DK=n\cdot \textbf{1}$.

        \item If $DK=n\cdot \textbf{1}$ has no solution, then $\mu^{\text{out}}_G$ is defined as 
        \begin{equation}
           \mu^{out}_G = D^\dagger (n\cdot \textbf{1}),
        \end{equation}
        where $D^\dagger$ is the Moore-Penrose inverse of $D$, which always exists and is unique.
    \end{enumerate}
    Similarly, the \textit{in-curvature} $\mu^{in}_G$ is defined by replacing $DK=n\cdot \textbf{1}$ with $D^TK=n\cdot \textbf{1}$ in the above definition.
\end{definition}

\begin{remark}
More explicitly, this means that for each $v\in V$, $\mu^{\text{out}}_G(w)$ and $\mu^{\text{in}}_G(w)$ must satisfy
\begin{equation}
\label{eq_outcurv}
    \sum_{w\in V} d(v,w)\mu^{\text{out}}_G(w) =n,
\end{equation}
and
\begin{equation}
\label{eq_incurv}
    \sum_{w\in V} d(w,v)\mu^{\text{in}}_G(w) =n,
\end{equation}
respectively. It is not difficult to see that if $\mu \in \R^{n}_{\geq0}$ is a nonnegative solution to $D\mu=n1_n$ (or $D^T\mu=n1_n$), then 
\[
\min_{v \in V} \mu(v) \leq \dfrac{n}{n-1}
\]
with equality if and only if $D=J_n$, the $n\times n$ matrix with all entries being one, which corresponds to $G=\overleftrightarrow{K_n}$. 
\end{remark}

As digraphs may be highly asymmetric, a single curvature notion is no longer sufficient. We therefore define two curvatures for digraphs, namely the in-curvature and the out-curvature. If the distance matrix $D$ is symmetric, so that $G$ can be identified with an undirected graph, then these two curvatures coincide, that is, 
$
\mu^\text{out}_G=\mu^\text{in}_G.
$
In this case, our definition reduces to the usual Steinerberger curvature. When the underlying graph $G$ is clear from the context, we shall write $\mu^{\text{out}}$ and  $\mu^{\text{in}}$ instead of $\mu^{\text{out}}_G$ and $\mu^{\text{in}}_G,$ respectively.


Throughout this section, we only consider digraphs for which both systems $DK=n \cdot \textbf{1}$ and $D^TK=n \cdot \textbf{1}$ admit solutions, corresponding to Cases (1) and (2) in Definition \ref{def_steinerbergercurv_directed}. In \cite{cushing2025note_steinerbeger}, the authors asked when the total Steinerberger curvature of an undirected graph can vanish or become negative. Motivated by this question, we provide sufficient and necessary conditions for a vertex of a digraph to possess negative in- and out-curvatures. The same argument can be extended naturally to obtain sufficient and necessary conditions for several vertices to have negative in- and out-curvatures.
\begin{proposition}
\label{prop_suff_negcurv}
    Let $G=(V,E)$ be a digraph with vertex set $V=\{v_1, \ldots, v_n\}$.
    Let the distance matrix of $G$ be expressed as 
    \begin{equation*}
    D= \left(
    \begin{matrix}
    0 & r_{n-1}^T \\ 
    c_{n-1} & D_{n-1}
    \end{matrix} \right),
    \end{equation*}
     where $r_{n-1}=d(v_1, \cdot ),  c_{n-1}=d(\cdot, v_1) \in \R^{n-1}$ and $D_{n-1} \in \R^{(n-1)\times(n-1)}$. Suppose that $r_{n-1}^T D_{n-1}^{-1}c_{n-1} \neq 0$ and $\det(D)\det(D_{n-1})<0$. Then, 
    \begin{enumerate}
        \item $v_1$ has negative out-curvature $\mu^{\text{out}}(v_1)<0$ if and only if $r_{n-1}^T D_{n-1}^{-1}\textbf{1}_{n-1} < 1$;
        
        \item $v_1$ has negative in-curvature $\mu^{\text{in}}(v_1)<0$ if and only if $c_{n-1}^T (D_{n-1}^T)^{-1}\textbf{1}_{n-1} < 1$.
    \end{enumerate}
\end{proposition}

\begin{proof}
    Let $\mu^{\text{out}}=(\mu^{\text{out}}(v_1), \mu^{\text{out}}_{n-1})$, where $\mu^{\text{out}}_{n-1}\in \R^{n-1}$. Solving for $D\mu^{\text{out}}=n\textbf{1}_n$, we have 
    \begin{align*}
        r^T_{n-1}\mu^{\text{out}}_{n-1}&=n, \\
        \mu^{\text{out}}(v_1)c_{n-1}+D_{n-1}\mu^{\text{out}}_{n-1}&=n\textbf{1}_{n-1}.
    \end{align*}
    By assumption, $D_{n-1}$ is invertible. Then,
         \[
         \mu^{\text{out}}(v_1)r^T_{n-1}D_{n-1}^{-1}c_{n-1}+r^T_{n-1}\mu^{\text{out}}_{n-1}=nr^T_{n-1}D_{n-1}^{-1}\textbf{1}_{n-1}, 
         \]
         which further simplifies to
         \begin{equation}\label{eq:prop4.3_eq1}
         \mu^{\text{out}}(v_1) 
         = n\dfrac{(r^T_{n-1}D_{n-1}^{-1}\textbf{1}_{n-1}-1)}{r^T_{n-1}D_{n-1}^{-1}c_{n-1}},
         \end{equation}
    which is well-defined following the assumption. By using the Schur complement on $D$, it follows that 
    \[
    \det(D)=-r^T_{n-1}D_{n-1}^{-1}c_{n-1}\det(D_{n-1}),
    \]
    giving $r^T_{n-1}D_{n-1}^{-1}c_{n-1}=-\dfrac{\det(D)}{\det(D_{n-1})}$. Since $\det(D)\det(D_{n-1}) <0$, we have 
    \[
    r^T_{n-1}D_{n-1}^{-1}c_{n-1}>0.
    \]
    Then, it follows from \eqref{eq:prop4.3_eq1} that $\mu^{\text{out}}(v_1)<0$ if and only if $r^T_{n-1}D_{n-1}^{-1}\textbf{1}_{n-1}<1.$
    

    The proof for negative in-curvature is similar. We omit the proof. 
\end{proof}

To illustrate Proposition \ref{prop_suff_negcurv}, we consider the following example.
\begin{example}
    Consider the directed cycle graph $C_5$ with vertex set $\{v_1, v_2,v_3,v_4,v_5\}$ and an extra arc $(v_1,v_4),$ as illustrated in Figure \ref{fig:C_5+arc}. The distance matrix is
    \begin{align*}
    D = \begin{pmatrix}
    0 & 1 & 2 & 1 & 2 \\
    4 & 0 & 1 & 2 & 3 \\
    3 & 4 & 0 & 1 & 2 \\
    2 & 3 & 4 & 0 & 1 \\
    1 & 2 & 3 & 2 & 1
    \end{pmatrix}, \quad 
    \text{with }
    D_{n-1} = \begin{pmatrix}
     0 & 1 & 2 & 3 \\
     4 & 0 & 1 & 2 \\
     3 & 4 & 0 & 1 \\
     2 & 3 & 2 & 1
    \end{pmatrix}.
    \end{align*}
    It is easy to verify that $r_{n-1}^T D_{n-1}^{-1} c_{n-1}=\dfrac{53}{18}\neq 0.$ 
    Also,
    \[
    \det D \det D_{n-1} = 265(-90)<0.
    \]
    Moreover, 
    $r_{n-1}^T
    =\begin{pmatrix}
       1 & 2 &1 &2 
    \end{pmatrix}$. By a simple computation, $r_{n-1}^TD_{n-1}^{-1}1_{n-1}=\dfrac{8}{9} <1$. By Proposition \ref{prop_suff_negcurv}, $\mu^{\text{out}}(v_1)<0$. Indeed, 
    \begin{align*}
    \mu^{\text{out}} =5D^{-1} 1_5 
    = \frac{1}{53} 
     \begin{pmatrix}
     -10 &
     30 &
     30 &
     25 &
     75
     \end{pmatrix}^T.
    \end{align*}
    implying $\mu^{\text{out}}(v_1)=-\dfrac{10}{53}<0$.
\end{example}

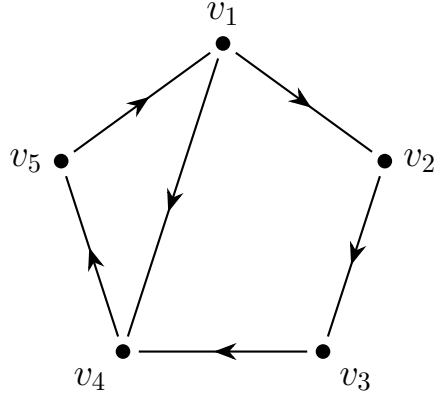
\begin{figure}[htp!]
    \centering
\begin{tikzpicture}[
    scale=1.5,
    vertex/.style={circle, fill=black, inner sep=0pt, minimum size=5.5pt},
    directed/.style={
        thick, black,
        shorten >=3pt, shorten <=3pt,
        postaction={decorate, decoration={markings, mark=at position 0.55 with {\arrow{Stealth[scale=1.3]}}}}
    }
]

    \node[vertex, label=above:{\large $v_1$}] (v1) at (90:1.5cm) {};
    \node[vertex, label=right:{\large $v_2$}] (v2) at (18:1.5cm) {};
    \node[vertex, label=below right:{\large $v_3$}] (v3) at (306:1.5cm) {};
    \node[vertex, label=below left:{\large $v_4$}] (v4) at (234:1.5cm) {};
    \node[vertex, label=left:{\large $v_5$}] (v5) at (162:1.5cm) {};

    \draw[directed] (v5) -- (v1); 
    \draw[directed] (v1) -- (v2); 
    \draw[directed] (v2) -- (v3); 
    \draw[directed] (v3) -- (v4); 
    \draw[directed] (v4) -- (v5); 

    \draw[directed] (v1) -- (v4); 

\end{tikzpicture}
    \caption{Directed cycle graph with added arc $(v_1,v_4)$}
    \label{fig:C_5+arc}
\end{figure}

The following proposition gives a sufficient condition on when the out-curvature is the permutation of the in-curvature.
\begin{proposition}
\label{prop_permutationofcurv}
   Let $G$ be a digraph. Suppose the distance matrix of $G$ is invertible and satisfies $PDQ=D^T$ for some permutation matrices $P$ and $Q$. Then, $\mu^{\text{out}}=Q\mu^{\text{in}}$.
\end{proposition}
\begin{proof}
    Suppose $PDQ=D^T$. From $D^T \mu^{\text{in}}=n\mathbf{1_n}$, we have
    \begin{align*}
        PDQ\mu^{\text{in}}=n\mathbf{1_n}.
    \end{align*}
    Left multiplication on both sides by $P^{-1}=P^T$ gives
      \begin{align*}
    DQ\mu^{\text{in}}=P^T(n\mathbf{1_n})=n\mathbf{1_n}.
    \end{align*}
    Thus, 
    \begin{align*}
        D\mu^{\text{out}}-DQ\mu^{\text{in}}=n\mathbf{1_n}-n\mathbf{1_n}=\mathbf{0_n}.
    \end{align*}
    Since $D$ is invertible, it follows that $\mu^{\text{out}}=Q\mu^{\text{in}}$.
\end{proof}

Thus, from Proposition \ref{prop_permutationofcurv}, one question naturally raised: \textit{is there a digraph in which the out-curvature is not a permutation of in-curvature}? We give a non-trivial example here.
\begin{example}
Consider the digraph modified from the undirected wheel graph $W_4$ with vertex set $V=\{v_1, v_2,v_3,v_4\},$ as illustrated in Figure \ref{fig:wheeldigraph}. It is immediate to observe that the distance matrix and its transpose are respectively given by
\begin{align*}
D = \begin{pmatrix}
0 & 1 & 2 & 1 \\
2 & 0 & 1 & 1 \\
1 & 2 & 0 & 1 \\
3 & 1 & 2 & 0
\end{pmatrix}, \quad 
D^T = \begin{pmatrix}
0 & 2 & 1 & 3 \\
1 & 0 & 2 & 1 \\
2 & 1 & 0 & 2 \\
1 & 1 & 1 & 0
\end{pmatrix}.
\end{align*}
Both $D$ and $D^T$ are invertible, thus the in- and out-curvature are respectively 
\begin{align*}
\mu^{\text{in}}=4(D^T)^{-1} 1_4 = \begin{pmatrix}
\dfrac{4}{3} &
\dfrac{4}{3} &
\dfrac{4}{3} &
0
\end{pmatrix}^T, \quad 
\mu^{\text{out}} =4D^{-1} 1_4= \begin{pmatrix}
\dfrac{2}{3} &
\dfrac{2}{3} &
\dfrac{2}{3} &
2
\end{pmatrix}^T.
\end{align*}

Interestingly, if we delete the arc $(v_2, v_4)$, then the in- and out-curvature become 
\begin{align*}
\mu^{\text{in}} = \begin{pmatrix}
\dfrac{6}{5} &
\dfrac{4}{5} &
\dfrac{6}{5} &
\dfrac{2}{5}
\end{pmatrix}^T, \quad 
\mu^{\text{out}} = \begin{pmatrix}
\dfrac{2}{5} &
\dfrac{6}{5} &
\dfrac{4}{5} &
\dfrac{6}{5}
\end{pmatrix}^T.
\end{align*}
So, we have $\mu^{\text{out}}=P\mu^{\text{in}}$ with the permutation matrix 
$P = \begin{pmatrix}
0 & 0 & 0 & 1 \\
1 & 0 & 0 & 0 \\
0 & 1 & 0 & 0 \\
0 & 0 & 1 & 0
\end{pmatrix}$.
\end{example}

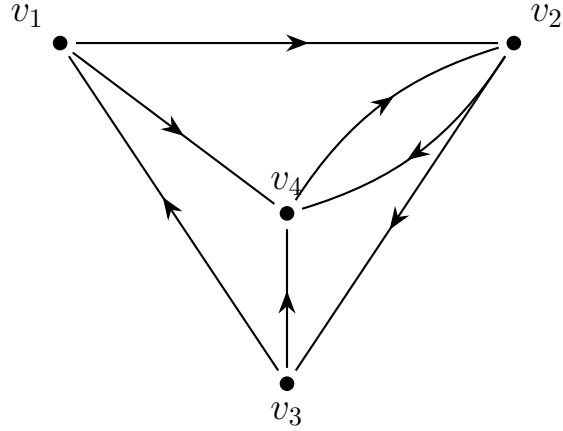
\begin{figure}[h]
    \centering
\begin{tikzpicture}[
    scale=1.5,
    vertex/.style={circle, fill=black!80!black, inner sep=0pt, minimum size=5.5pt},
    directed/.style={
        thick, black!80!black,
        shorten >=3pt, shorten <=3pt,
        postaction={decorate, decoration={markings, mark=at position 0.55 with {\arrow{Stealth[scale=1.3]}}}}
    }
]

\node[vertex, label=above left:{\large \color{black!80!black}$v_1$}] (v1) at (-2, 1.5) {};
\node[vertex, label=above right:{\large \color{black!80!black}$v_2$}] (v2) at (2, 1.5) {};
\node[vertex, label=below:{\large \color{black!80!black}$v_3$}] (v3) at (0, -1.5) {};
\node[vertex, label=above:{\large \color{black!80!black}$v_4$}] (v4) at (0, 0) {};

\draw[directed] (v1) -- (v2); 
\draw[directed] (v2) -- (v3); 
\draw[directed] (v3) -- (v1); 

\draw[directed] (v1) -- (v4); 
\draw[directed] (v3) -- (v4); 

\draw[directed] (v4) to[bend left=20] (v2);  
\draw[directed] (v2) to[bend left=20] (v4);  

\end{tikzpicture}
    \caption{Modified Wheel Digraph}
    \label{fig:wheeldigraph}
\end{figure}

Let $\|\mu^{\text{out}}\|_{\ell_1}=\sum_{w \in V} \mu^{\text{out}}(w)$ and $\|\mu^{\text{in}}\|_{\ell_1}=\sum_{w \in V} \mu^{\text{in}}(w)$. From Proposition \ref{prop_permutationofcurv}, it is immediate that the out and in curvature have equal norm for digraphs satisfying $PDQ=D^T$. The following proposition generalizes the result.
\begin{proposition}
\label{prop_equal_curv}
    For a digraph $G$ with in- and out-curvature $\mu^{\text{out}}$ and $\mu^{\text{in}}$ satisfying Equations \eqref{eq_outcurv} and \eqref{eq_incurv}, it holds that $\|\mu^{\text{out}}\|_{\ell_1}=\|\mu^{\text{in}}\|_{\ell_1}$.
\end{proposition}

\begin{proof}
    For any $x,y\in \R^n$, let $\<x,y \>=x^Ty$ be the inner product of $x$ and $y$. Then, 
    \[
        \|\mu^{\text{out}}\|_{\ell_1}
        =\<\mu^{\text{out}}, \textbf{1} \>
        =\<\mu^{\text{out}},\frac{1}{n}D^T\mu^{\text{in}}\>
        =\<\frac{1}{n} D\mu^{\text{out}},\mu^{\text{in}}\>
        =\<\textbf{1},\mu^{\text{in}}\>=\|\mu^{\text{in}}\|_{\ell_1}.  \qedhere
    \]
\end{proof}


We shall call $\|\mu_G^{\text{out}}\|_{\ell_1}$ the total Steinerberger curvature of the digraph $G$. Next, we  establish several comparison-type results for this directed setting, including discrete analogues of the Bonnet--Myers theorem, Cheng's theorem, the Lichnerowicz theorem, and the reverse Bonnet--Myers inequality. Steinerberger \cite{steinerberger2023curvature} used the von Neumann Minimax theorem to prove these results in the undirected case. In our approach, the minimax theorem is used to derive the reverse Bonnet--Myers inequality, while the Bonnet--Myers, Cheng's and Lichnerowicz theorems are proved directly.


A general statement of von Neumann Minimax theorem is presented here.

\begin{theorem}\cite[von Neumann Minimax Theorem]{Neumann1928games}
\label{theo_vonneumann}
    Let 
    \[
    X=\left\{ x=(x_1, \ldots, x_n)\in \R_{\geq0}^n : \sum_i x_i=1 \right\}.
    \]
    For any matrix $A\in \R^{n \times n}$, there exists $\alpha\in \R$ such that 
    \begin{equation}
    \label{eq_minmax_alpha}
        \min_{x\in X} \max_{1\leq i\leq n} (Ax)_i =\alpha = \max_{x\in X} \min_{1\leq i\leq n} (A^Tx)_i.  
    \end{equation}
\end{theorem}
A variational theorem of von Neumann Minimax Theorem is readily obtained as follows.
\begin{theorem}
\label{theorem_variation}
    Suppose $G=(V,E)$ admits nonnegative in- and out-curvature $\mu^{\text{in}}, \mu^{\text{out}}$. For any probability measure $\eta :V \rightarrow [0,1]$, we have 
     \begin{equation}
    \label{ineq_minmax_var}    
      \min_{v\in V} \sum_{w\in V} d(w,v)\eta(w) \leq \dfrac{n}{\|\mu^{\text{in}}\|_{\ell_1}} \leq \max_{v\in V} \sum_{w\in V} d(v,w)\eta(w).  
    \end{equation}
\end{theorem}

\begin{proof}
     In Theorem \ref{theo_vonneumann}, consider $A=D$ and $x=\dfrac{\mu^{\text{out}}}{\|\mu^{\text{out}}\|_{\ell_1}}$. From the left equality of \eqref{eq_minmax_alpha}, we have 
    \begin{align*}
    \alpha &\leq \max_{v \in V} \left(\sum_{w\in V} d(v,w)\dfrac{\mu^{\text{out}}(w) }{\|\mu^{\text{out}}\|_{\ell_1}} \right) = \frac{n}{\|\mu^{\text{out}}\|_{\ell_1}}.
    \end{align*}
    Now by choosing $x$ to be $\dfrac{\mu^{\text{in}}}{\|\mu^{\text{in}}\|_{\ell_1}}$, the right equality of \eqref{eq_minmax_alpha} gives  $\alpha \geq  \dfrac{n}{\|\mu^{\text{in}}\|_{\ell_1}}$.

    Since $\|\mu^{\text{out}}\|_{\ell_1}=\|\mu^{\text{in}}\|_{\ell_1}$ by Proposition \ref{prop_equal_curv}, we conclude that $\alpha = \dfrac{n}{\|\mu^{\text{in}}\|_{\ell_1}}$.
\end{proof}

Using Theorem \ref{theorem_variation}, we prove an analogue of Reverse Bonnet-Myers Inequality for digraphs. We follow the argument in \cite[Theorem 2]{steinerberger2023curvature}.

\begin{theorem}[Reverse Bonnet-Myers Inequality]
\label{theo_reverse_bonnet}
        Let $G=(V,E)$ be a digraph admitting nonnegative in- and out-curvature $\mu^{\text{in}}, \mu^{\text{out}}$. It holds that
        \begin{equation}
        \label{eq:rev_BM_ineq}
        \overrightarrow{\diam}(G) \geq \dfrac{n^2}{(n-1)} \dfrac{1}{\|\mu^{\text{in}}\|_{\ell_1}},
        \end{equation}
        with equality if and only if $G=\overleftrightarrow{K_n}$.
\end{theorem}
\begin{proof}
     Consider the uniform measure $\eta \equiv \frac{1}{n}$. By Theorem \ref{theorem_variation}, there exists $v_0$ such that $\dfrac{n}{\|\mu^{\text{in}}\|_{\ell_1}} \leq \dfrac{1}{n}\sum_{w\in V}d(v_0,w)$. Since $\sum_{w\in V}d(v_0,w) \leq (n-1)\overrightarrow{\diam}(G),$ we immediately obtain \eqref{eq:rev_BM_ineq}.
    Suppose the equality holds. Then, it implies that $d(v_0,w)=\overrightarrow{\diam}(G)$ for every $w\neq v_0$. If $w$ is an out-neighborhood of $v_0$, then $\overrightarrow{\diam}(G)=d(v_0,w)=1$. This implies that $G$ is $\overleftrightarrow{K_n}$. The converse direction is easy to check.
\end{proof}

In the following, we consider a \textit{generalized} reverse Bonnet-Myers Inequality formulated in terms of out-radius.

\begin{theorem}
\label{theo:reversebonnet_outradius}
    Let $G=(V,E)$ be a digraph. Suppose $G$ admits nonnegative out-curvature $\mu^{\text{out}}$. Then,
    \begin{equation} \label{eq:generalized_rev_BM_ineq}
    \mathrm{rad}^+(G) \geq \dfrac{n}{\|\mu^{\text{out}}\|_{\ell_1}}.
    \end{equation}
    Let $v_0$ be a vertex such that $\text{dist}(v_0, V)=\mathrm{rad}^+(G)$. If in addition $\mu^{\text{out}}(v_0) \geq \dfrac{\|\mu^{\text{out}}\|_{\ell_1}}{n}$ holds, then
    \begin{equation*}
       \mathrm{rad}^+(G) \geq \dfrac{n^2}{n-1} \dfrac{1}{\|\mu^{\text{out}}\|_{\ell_1}},
    \end{equation*}
     with equality if only if $d(v_0,w)=1$ for every $w\neq v_0$ (Note that this implies that $\mathrm{rad}^+(G)=1$ and $\|\mu^{\text{in}}\|_{\ell_1}=\frac{n^2}{n-1}$).
\end{theorem}

\begin{proof}
    From the definition of $\mu^{\text{out}}$, we have
    \begin{equation}\label{eq:thm4.11_ineq1}
    n=\sum_{w\in V} d(v_0, w) \mu^{\text{out}}(w) 
         \leq d(v_0, V) \sum_{w\in V-\{v_0\}} \mu^{\text{out}}(w) 
         =\mathrm{rad}^+(G) (\|\mu^{\text{out}}\|_{\ell_1}-\mu^{\text{out}}(v_0)).
    \end{equation}
    Since $\mu^{\text{out}}$ is assumed to be nonnegative, we must have $\mu^{\text{out}}(v_0)\geq 0$. Then, \eqref{eq:generalized_rev_BM_ineq} follows immediately from \eqref{eq:thm4.11_ineq1}.
    If, in addition, we assume that $\mu^{\text{out}}(v_0) \geq \dfrac{\|\mu^{\text{out}}\|_{\ell_1}}{n}$, then  \eqref{eq:thm4.11_ineq1} implies that $\mathrm{rad}^+(G) \geq \dfrac{n^2}{n-1} \dfrac{1}{\|\mu^{\text{out}}\|_{\ell_1}}$.
\end{proof}

\begin{theorem}
\label{theo:inradius_ub}
    Let $G=(V,E)$ be a digraph with nonnegative out-curvature $\mu^{\text{out}}$. Then, the inequality 
    \begin{equation}\label{eq:in_rad_ineq}
    \mathrm{rad}^-(G) \leq \dfrac{n^2}{\|\mu^{\text{in}}\|_{\ell_1}} -(n-2)
    \end{equation}
    holds, with equality only if there exist $v_0,w_0 \in V$ such that $d(w,v_0)=1$ for every $w \neq v_0, w_0$ and $d(w_0,v_0)=\mathrm{rad}^{-}(G)$. 
\end{theorem}

\begin{proof}
    In Theorem \ref{theorem_variation}, consider the uniform measure $\eta \equiv \frac{1}{n}$. Then there exists $v_0\in V$ such that 
    \begin{equation*}
        \frac{1}{n}\sum_{w\in V} d(w,v_0) \leq \dfrac{n}{\|\mu^{\text{in}}\|_{\ell_1}}.
    \end{equation*}
    Choose $w_0 \in V$ such that $\text{dist}(V,v_0)=d(w_0,v_0)$. Then, 
        \[
        \dfrac{n}{\|\mu^{\text{in}}\|_{\ell_1}} 
        \geq \frac{1}{n}d(w_0,v_0) + \frac{1}{n} \sum_{w\in V-\{v_0, w_0\}} d(w,v_0) 
        \geq \frac{1}{n}\text{dist}(V,v_0) + \frac{1}{n} (n-2). 
        \]
        This implies that
        \[
        \text{dist}(V,v_0) \leq \dfrac{n^2}{ \|\mu^{\text{in}}\|_{\ell_1}} -(n-2).
        \]
    By Definition \ref{def2.3}(ii), the Inequality \eqref{eq:in_rad_ineq} follows immediately.
    Suppose the equality holds, i.e. $\mathrm{rad}^-(G) = \dfrac{n^2}{\|\mu^{\text{in}}\|}_{\ell_1} -(n-2)$. Tracing back the inequalities, this implies that $\mathrm{rad}^{-}(G)=\text{dist}(V,v_0)=d(w_0,v_0)$ and $d(w,v_0)=1$ for every $w \neq v_0,w_0$.
\end{proof}

Following the idea in \cite{cushing2025note_steinerbeger}, we present a direct proof of Discrete Bonnet-Myers and Cheng's theorem for digraphs. 
\begin{theorem}[Discrete Bonnet-Myers Theorem]
\label{theo:bonnet-myers}
    Let $G=(V,E)$ be a digraph. Suppose $G$ has in- and out-curvatures bounded below by $K>0$, i.e. $\min_v \mu^{\text{out}}(v),$ $\min_v \mu^{\text{in}}(v) \geq K>0$. Then, it holds that
    \begin{equation}
        \overrightarrow{\diam (G)} \leq \dfrac{3n}{\|\mu^{\text{in}}\|_{\ell_1}}-\dfrac{2(\mu^{\text{in}})^TA(\mu^{\text{out}})}{\|\mu^{\text{in}}\|_{\ell_1}^2} \leq \frac{3}{K},
    \end{equation}
    where $A=\frac{1}{2}(D-D^T)$ is the skew-symmetric part of $D$. \vspace{7pt}

    (Cheng's Theorem) If $\overrightarrow{\diam (G)}=\dfrac{3}{K}$, then $G$ has constant in- and out-curvature $K\mathbf{1_n}$. 
\end{theorem}

\begin{proof}
    Let $(x,y)$ be an arc with $d(x,y)= \overrightarrow{\diam (G)}$. By triangle inequality, for any $v,w\in V$, 
    \begin{align*}
         \overrightarrow{\diam (G)} =d(x,y) \leq d(x,v)+d(v,w)+d(w,y).
    \end{align*}
    Multiplying both sides by $\mu^{\text{out}}(v)\mu^{\text{in}}(w)$ and summing over all $v,w\in V$, we get
    \begin{align*}
         \overrightarrow{\diam (G)} \sum_{v,w\in V} \mu^{\text{out}}(v)\mu^{\text{in}}(w) &\leq \sum_{v,w\in V} \mu^{\text{out}}(v)\mu^{\text{in}}(w) d(x,v)  +  \sum_{v,w\in V} \mu^{\text{out}}(v)\mu^{\text{in}}(w)d(v,w) \\
         &+  \sum_{v,w\in V} \mu^{\text{out}}(v)\mu^{\text{in}}(w)d(w,y) \\
         &=n \|\mu^{\text{in}}\|_{\ell_1}+ (\mu^{\text{out}})^T D (\mu^{\text{in}}) + n \|\mu^{\text{out}}\|_{\ell_1} \\
         &=2n\|\mu^{\text{in}}\|_{\ell_1}+(\mu^{\text{out}})^T D (\mu^{\text{in}}).
    \end{align*}
    Let $A=\frac{1}{2}(D-D^T)$ be the skew-symmetric part of $D$. Then, 
    \begin{align*}
        2n\|\mu^{\text{in}}\|_{\ell_1}+(\mu^{\text{out}})^T D (\mu^{\text{in}}) &= 2n\|\mu^{\text{in}}\|_{\ell_1}+(\mu^{\text{out}})^T (2A+D^T) (\mu^{\text{in}}) \\
        &= 3n\|\mu^{\text{in}}\|_{\ell_1} -2(\mu^{\text{in}})^TA(\mu^{\text{out}})
    \end{align*}
    Hence, $\overrightarrow{\diam (G)} \leq \dfrac{3n}{\|\mu^{\text{in}}\|_{\ell_1}}-\dfrac{2(\mu^{\text{in}})^TA(\mu^{\text{out}})}{\|\mu^{\text{in}}\|_{\ell_1}^2}$. By assumption, $\min_v \mu^{\text{out}}(v),\min_v \mu^{\text{in}}(v) \geq K>0$, we have $\|\mu^{\text{in}}\|_{\ell_1} \geq nK$ and 
    \begin{align*}
         2(\mu^{\text{in}})^TA(\mu^{\text{out}})&=\sum_{v,w\in V} \mu^{\text{in}}(v) (d(v,w)-d(w,v)) \mu^{\text{out}}(w) \\
         &\geq K^2\sum_{v,w\in V} (d(v,w)-d(w,v)) 
         =0.
    \end{align*} 
    The last equality follows as the sum of all entries of skew-symmetric part of any matrix is zero. Therefore,
    \vspace{-0.4cm} 
    \[
        \overrightarrow{\diam (G)} \leq \dfrac{3n}{\|\mu^{\text{in}}\|_{\ell_1}}-\dfrac{2(\mu^{\text{in}})^TA(\mu^{\text{out}})}{\|\mu^{\text{in}}\|_{\ell_1}^2} \leq \frac{3n}{nK}-0=\frac{3}{K}. 
    \]
    Finally, suppose $\overrightarrow{\diam(G)}=\dfrac{3}{K}$. Tracing back the inequalities, we see that 
    \[
    \min_{v} \mu^{\text{out}}(v)=\min_{v} \mu^{\text{out}}(v)=K \ \text{ and } \ \|\mu^{\text{in}}\|_{\ell_1}=\|\mu^{\text{out}}\|_{\ell_1}=nK.
    \] 
    Together, this implies that $\mu^{\text{out}}=\mu^{\text{out}}=K\mathbf{1_n}$. Therefore, $G$ has constant in- and out-curvature $K\mathbf{1_n}$. 
    \end{proof}

    \vspace{-0.2cm}    
We shall now prove the Discrete Lichnerowicz Theorem for the first nonzero eigenvalue of the Laplacian of digraphs. 
\begin{theorem}[Discrete Lichnerowicz Theorem]
\label{theo:lichnerowicz}
Let $G=(V,E)$ be a digraph. Suppose $G$ has out-curvature bounded below by $K>0$, i.e. $\min_v \mu^{\text{out}}(v) \geq K>0$. Then, 
   \begin{equation}
        \lambda_1 \geq \frac{\min_{(u,v)\in E} \phi(u)P(u,v)}{2n} 
        \left( \|\mu^{\text{out}}\|_{\ell_1} \ + \min_v \mu^{\text{out}}(v) \right) \geq \frac{\min_{(u,v)\in E} \phi(u)P(u,v)}{2n} (n+1)K. 
    \end{equation}
\end{theorem}
\begin{proof}
    Suppose that $f:V \rightarrow \R$ is a function that achieves the infimum in Equation \eqref{eq_lambda_1_inf}. Consider any arc $(x,y)\in E$. Let $\gamma_{xy}$ be a shortest path from $x$ to $y$. Then
    \begin{align*}
        (f(x)-f(y))^2 &= \left( \sum_{(u,v)\in \gamma_{xy}} (f(u)-f(v))   \right)^2 \\
        &\leq d(x,y)\sum_{(u,v)\in \gamma_{xy}} (f(u)-f(v))^2.
    \end{align*}
    Multiply both sides by $\min_{(u,v)\in \gamma_{xy}} \phi(u)P(u,v)>0$, we have
    \begin{align*}
       \min_{(u,v)\in \gamma_{xy}} \phi(u)P(u,v) (f(x)-f(y))^2 &\leq d(x,y) \min_{(u,v)\in \gamma_{xy}} \phi(u)P(u,v)\sum_{(u,v)\in \gamma_{xy}} (f(u)-f(v))^2 \\
       &\leq   d(x,y)\sum_{(u,v)\in \gamma_{xy}} (f(u)-f(v))^2 \phi(u)P(u,v) \\
       &\leq  d(x,y) \sum_{(u,v)\in E} (f(u)-f(v))^2 \phi(u)P(u,v).
    \end{align*}
    Let    
    \vspace{-0.4cm} 
    \begin{align*}
    \alpha(f) &=\dfrac{\sum_{(u,v)\in E} (f(u)-f(v))^2 \phi(u)P(u,v)}{2},   \\
    \beta &=\min_{(u,v)\in \gamma_{xy}} \phi(u)P(u,v) \geq \min_{(u,v)\in E} \phi(u)P(u,v).
    \end{align*}
    Multiplying both sides of the inequalities by $\phi(x)\mu^{\text{out}}(y)$ and sum over all $x,y\in V$, we have
    \[
    \beta \sum_{x,y\in V} (f(x)-f(y))^2 \phi(x)\mu^{\text{out}}(y)  \leq 2\alpha(f) \sum_{x,y\in V} d(x,y)\phi(x)\mu^{\text{out}}(y). 
    \]
    Expanding the left side and simplifying the right side,
    \[
    \beta \sum_{x,y\in V} \phi(x)\mu^{\text{out}}(y) [f^2(x)-2f(x)f(y)+f^2(y)] \leq 2n\alpha(f) .
    \]
    It follows that
    \[
    \|\mu^{\text{out}}\|_{\ell_1} \<\phi, f^2 \> + \sum_{y\in V} f^2(y)\mu^{\text{out}}(y)   \leq \frac{2n\alpha(f)}{\beta}.
    \]
    We focus on the term $\sum_{y\in V} f^2(y)\mu^{\text{out}}(y)$. Note that 
    \begin{align*}
        \sum_{y\in V} f^2(y)\mu^{\text{out}}(y) &\geq \min_v \mu^{\text{out}}(v)   \sum_{y\in V} f^2(y)\\
        &=\min_v \mu^{\text{out}}(v)  \sum_{x\in V} \phi(x) \sum_{y\in V} f^2(y) \\
        &\geq\min_v \mu^{\text{out}}(v)  \sum_{y\in V} \phi(y)f^2(y) \\
        &=  \min_v \mu^{\text{out}}(v)  \<\phi, f^2 \>.
    \end{align*}
    Hence, we have
        \[
         \|\mu^{\text{out}}\|_{\ell_1} \<\phi, f^2 \> + \min_v \mu^{\text{out}}(v)  \<\phi, f^2 \>   \leq \frac{2n\alpha(f)}{\beta} ,
         \]
         which yields
         \[
         \frac{\beta}{2n} \left( \|\mu^{\text{out}}\|_{\ell_1} \ + \min_v \mu^{\text{out}}(v)  \right) \leq \frac{\alpha(f)}{\<\phi, f^2 \>}
         \]
    Therefore, we arrive at 
    \begin{equation}
        \lambda_1 \geq \frac{\min_{(u,v)\in E} \phi(u)P(u,v)}{2n} 
        \left( \|\mu^{\text{out}}\|_{\ell_1} \ + \min_v \mu^{\text{out}}(v)  \right).
    \end{equation}
    Since $\min_v \mu^{\text{out}}(v), \min_v \mu^{\text{in}}(v) \geq K>0$, we obtain
    \begin{equation*}
        \lambda_1 \geq \frac{\min_{(u,v)\in E} \phi(u)P(u,v)}{2n} 
        \left( \|\mu^{\text{out}}\|_{\ell_1} \ + \min_v \mu^{\text{out}}(v)  \right) \geq \frac{\min_{(u,v)\in E} \phi(u)P(u,v)}{2n} (n+1)K .
    \end{equation*}
\end{proof}
\begin{remark}
    \begin{enumerate}
        \item If $D$ is symmetric, then $\cL=I_n-D^{-\frac{1}{2}}AD^{-\frac{1}{2}}$ is the \textit{normalized} Laplacian matrix of $G$. Meanwhile $\lambda_1$ in \cite[Theorem 3]{steinerberger2023curvature} is the first nonzero eigenvalue of the \textit{unnormalized} Laplacian $L=D-A$. 
        \item When $G$ is undirected, then $\phi(u)=\dfrac{d_u}{\text{vol}(G)}$ and $P(u,v)=\dfrac{1}{d_u}$, so  
        
        \[
        \min_{(u,v)\in E} \phi(u)P(u,v) = \dfrac{1}{\text{vol}(G)},
        \] where $\text{vol}(G)=\sum_{v\in V} d_v$. Therefore, 
        \[
        \lambda_1 \geq \dfrac{(n+1)K}{2n\text{vol}(G)}.
        \] 
        Hence, if $G$ is $r$-regular, then $\text{vol}(G)=nr$, so 
        $ \lambda_1 \geq \dfrac{(n+1)K}{2n^2r}$. Let $\lambda'_1$ be the first nonzero eigenvalue of the unnormalized Laplacian $L$. Then, $\lambda_1' =r\lambda_1$. Our lower bound becomes 
        \[
        \lambda_1' \geq \dfrac{(n+1)K}{2n^2}>\dfrac{K}{2n}.
        \]
        This shows that the lower bound $ \dfrac{(n+1)K}{2n^2}$ is better than the lower bound $\dfrac{K}{2n}$ in \cite[Theorem 3]{steinerberger2023curvature} for the case of undirected regular graphs.
    \end{enumerate}    
\end{remark}


\section*{Acknowledgment}
Johnny Lim acknowledges the support from the Ministry of Higher Education Malaysia for Fundamental Research Grant Scheme with Project Code: 
\linebreak 
FRGS/1/2025/STG06/USM/02/1. \\


\noindent \textbf{Conflicts of interest.} \quad The authors declare no conflicts of interest. \\
\textbf{Data availability.} \quad Not applicable.

\bibliography{bibliography}{}
\bibliographystyle{amsplain}



\end{document}